\documentclass[11pt,fleqn]{amsart}
\usepackage{amsmath, amsthm, amssymb, amsfonts, verbatim,color,esint}
\usepackage[pdftex]{graphicx}
\usepackage{csquotes}
\usepackage[bookmarks]{hyperref}

\newcommand{\R}{{\mathbb R}}
\newcommand{\e}{{\varepsilon}}
\newcommand{\eps}{{\varepsilon}}
\newcommand{\tor}{{\mathbb T^2}}
\DeclareMathOperator{\dist}{dist}

\newtheorem{theorem}{Theorem}[section]
\newtheorem{corollary}[theorem]{Corollary}
\newtheorem{lemma}[theorem]{Lemma}

\theoremstyle{definition}

\theoremstyle{remark}

\title[Generation of vortices in the Ginzburg-Landau  heat flow]{Generation of vortices in the Ginzburg-Landau heat flow
 }

\author{Micha{\l } Kowalczyk}
\address{Departamento de Ingenier\'{\i}a Matem\'atica and Centro
de Modelamiento Matem\'atico (UMI 2807 CNRS), Universidad de Chile, Casilla
170 Correo 3, Santiago, Chile
\and
Institute of Mathematics, Polish Academy of Sciences, ul. \'Sniadeckich 8, 00-656, Warsaw, Poland}
\email {kowalczy@dim.uchile.cl}
\thanks{M. K. was partially funded by Chilean research grants FONDECYT 1210405 and ANID projects ACE210010 and FB210005 and National Science Centre, Poland (Grant
No. 2020/37/B/ST1/02742).  Part of this work was done during his stay in the  Institut de Math\'ematiques de Toulouse. }

\author{Xavier Lamy}
\address{Institut de Math\'ematiques de Toulouse; UMR 5219, Universit\'e de Toulouse; CNRS, UPS IMT, F-31062 Toulouse Cedex 9, France.}
\email{Xavier.Lamy@math.univ-toulouse.fr}
\thanks{X.L. was partially funded by ANR project ANR-22-CE40-0006-01.
Part of this work was conducted during X.L.'s stay at the CMM, with the financial support of CNRS and the CMM}

\begin{document}

\begin{abstract}
We consider the Ginzburg-Landau heat flow on the two-dimensional flat torus, starting from an initial data with a finite number of nondegenerate zeros -- but possibly very high initial energy.
We show that the initial zeros are conserved and the flow rapidly enters a logarithmic energy regime, from which the evolution of vortices can be described by the works of Bethuel, Orlandi and Smets.
\end{abstract}

\maketitle

\section{Introduction}

In the flat two-dimensional  torus $\mathbb T^2=\R^2/\mathbb Z^2$ we consider $u(t,x)$, a solution of the Ginzburg-Landau heat flow
\begin{equation}
\label{eq:gl 1}
\begin{aligned}
\partial_t u -\e^2\Delta u &=(1-|u|^2)u\qquad t\geq 0, \; x\in\tor,\\
u(0,x)&=u_0(x),
\end{aligned}
\end{equation}
with $u_0\in C^1(\tor)$.
The initial condition $u_0$ may have a finite number of zeros.
More precisely, we assume that there exists $\alpha_0>0$ such that 
\begin{equation}
\label{eq:u0nondegen}
|u_0(x)|+|\det\nabla u_0(x)|\geq \alpha_0.
\end{equation}
This implies in particular that the zeros of $u_0$ are nondegenerate and the topological  degree of the vector field $u_0$ at each zero is $1$ or $-1$. 

We will denote the energy associated with \eqref{eq:gl 1} by
\[
E_\eps(u)=\int_{\tor}|\nabla u|^2+\frac{1}{4\eps^2}(1-|u|^2)^2.
\] 
Note that \eqref{eq:gl 1} is the $L^2$ gradient flow of $E_\eps$ up to a factor $\eps^2$, hence $E_\eps$ is decreasing along the flow. 
The Ginzburg-Landau heat flow has been extensively studied \cite{BCPS95,JS98-dyn,lin96,SS04,ser07-I,ser07-II,BOS05-coll,BOS07-dyn,BOS07-quant}, in the case of initial data $u_0=u_{0\e}$ satisfying a logarithmic energy bound $E_\e(u_{0\e})\leq M\ln(1/\e)$. 
This bound enables  to identify vortices, the zeros of $u_{0\e}$, and to describe their evolution. 
More precisely, in \cite{JS98-dyn,lin96,SS04}, 
well-prepared initial data are considered, with  a finite number of vortices of degree $\pm 1$ and correspondingly quantized energy.
These works establish via different methods that, in the accelerated time-scale $s=(\e^{2}/\ln(1/\e))t$, vortices move according to the gradient flow of a renormalized energy analyzed in \cite{BBH}, for as long as no collisions happen.
This limitation is removed in the works \cite{ser07-II,BOS07-dyn,BOS07-quant}, where splittings and collisions of vortices are described rigorously.
Specifically, \cite{ser07-II} describes the global-in-time motion of vortices, taking collisions into account, in bounded domains with Dirichlet or Neumann boundary conditions.
Initial well-preparedness is also relaxed: initial vortices are of degree $\pm 1$, but the energy quantization assumption is less stringent; moreover, splitting of higher degree vortices into vortices of degree $\pm 1$ is  described under specific assumptions.
In \cite{BOS07-dyn,BOS07-quant}, the domain is the whole plane and a global motion law allowing for splittings and collisions is obtained, for initial data satisfying the logarithmic bound $E_\e(u_{0\e})\leq M\ln(1/\e)$. 
In case of  $N_\e\gg 1$ initial vortices, evolution of the vortex density is described by a mean-field equation first obtained rigorously in \cite{ser17}.

Here we are interested in initial data that may have much higher energy, and wish to describe the emergence of vortices.
This is mentioned as an open problem in \cite[Problem~5]{BOS08-survey}.
Our methods are strongly inspired by similar results on the emergence of sharp transitions in the Allen-Cahn heat flow \cite{chen04}.

Our first main result concerns the evolution of the zeros of $u$.

\begin{theorem}\label{t:zeros}
There exists $C_0>0$, depending on $u_0$, such that, for all $\e>0$ sufficiently small (depending on $u_0$), if $\mathcal Z(t)$ denotes the set of zeros of $u(t)$, we have
\begin{align*}
\#\mathcal Z(t)=\#\mathcal Z(0), \qquad\text{for } 0\leq t\leq T_\e:=\ln\frac 1\e -\frac 12\ln\ln\frac 1\e - C_0.
\end{align*}
In other words, no new zeros of the vector field $u(t)$ are generated up to $t=T_\eps$. Additionally, if $z_j(t)$ is the evolution of the $j$-th zero $z_{j}^0$ of $u_0$, then $|z_j(t)-z_j^0|\lesssim \e\sqrt{\ln(1/\e)}$, and the topological  degree   of  $u(t)$ at $z_j(t)$ is preserved.
Finally, at $t=T_\e$ we have
\begin{align}\label{eq:baddisksTe}
|u(T_\e,x)|\geq \frac 12 \qquad\text{for }\dist(x,\mathcal Z(0))\gtrsim \e\sqrt{\ln \frac 1\e}.
\end{align}
\end{theorem}

Above and throughout the paper  the symbol $A\lesssim B$ for two nonnegative quantities $A,B$ means that there exists a constant $C>0$, depending only on $u_0$,  such that $A\leq C B$.

An immediate corollary of   Theorem~\ref{t:zeros} is that, if $u_0$ does not vanish, then $u(t)$ does not vanish for $0\leq t\leq T_\e$.

\begin{corollary}\label{c:nozero}
If $\mathcal Z(0)=\emptyset$ then $\mathcal Z(t)=\emptyset$ for  $t\in [0,T_\eps]$.
\end{corollary}

This means that up to time $T_\eps$ the Ginzburg-Landau heat flow does not undergo a  Berezinsky-Kosterlitz-Thouless  phase transition. 
If one allows the initial condition $u_0$ to depend on $\e$, one may however observe creation of zeros, as in \cite[Proposition~4.1]{RS95}.

 Our second main result is a logarithmic energy bound at the time $t=T_\e$ given by Theorem~\ref{t:zeros}.
\begin{theorem}\label{t:energy}
For all sufficiently small $\eps >0$ (depending on $u_0$), we have
\begin{align*}
E_\e\left(u(t)\right)\lesssim \ln\frac{1}{\e},\qquad \forall t\geq T_\eps.
\end{align*}
\end{theorem}

Theorem~\ref{t:energy} shows that the evolution enters an energy regime where the analysis of \cite{BOS05-coll,BOS07-dyn,BOS07-quant} can be applied. 
The present context is actually slightly different, because we work on the torus $\mathbb T^2$ instead of $\R^2$, but the results of \cite{BOS05-coll,BOS07-dyn,BOS07-quant} should apply to $\mathbb T^2$, with appropriate modifications.
Reciprocally,
 the results of the present paper could be adapted to $\R^2$,
  with appropriate conditions at infinity, at the price of minor technical complications.

In particular, the work \cite{BOS07-quant} describes the evolution of the vortices of $u$ as functions of the accelerated time-variable
\begin{align*}
s=\frac{\e^2}{\ln\frac 1\e} \, t.
\end{align*}
The vortices $a_k(s)$ evolve according to the gradient flow of a renormalized energy $W(a)$, combined with a finite number of collision or branching times.
 Note that in the torus $\mathbb T^2$, the renormalized energy $W(a)$ would be slightly different than the one considered in \cite{BOS07-quant}. 
The initial conditions for the vortices $a_k(s)$ as $s\to 0^+$ are identified via the  jacobian $Ju=\det(\nabla u)$ at the initial time \cite[Proposition~2]{BOS08-survey}. 
We therefore complement Theorem~\ref{t:energy} with our third main result, which characterizes the jacobian  at time $t=T_\e$.

\begin{theorem}\label{t:jac}
We have, as $\e\to 0$,
\begin{align*}
J u(T_\e)=\det(\nabla u(T_\e))\longrightarrow \sum_{j=1}^N \hat d_j \delta_{z_j^0},
\end{align*}
in the sense of distributions, where $z_1^0,\ldots,z_N^0$ are the zeros of $u_0$, and \mbox{$\hat d_j\in\lbrace \pm 1\rbrace$} its topological degree at $z_j^0$.
\end{theorem}

Now we can  be more specific about the initial conditions for the later evolution of the vortices $a_k(s)$, as described in \cite[Proposition~2]{BOS08-survey}. Letting $d_k$ denote the topological degree of $u(s)$ at $a_k(s)$ for small $s>0$, the initial conditions $a_k^0=\lim_{s\to 0^+}a_k(s)$ must satisfy
\begin{align*}
\sum_{k=1}^L d_k \delta_{a_k^0} =\sum_{j=1}^N \hat d_j \delta_{z_j^0}.
\end{align*} 
This implies in particular that $\lbrace a_k^0\rbrace =\lbrace z_j^0\rbrace$. But the points $a_k^0$ may not be disjoint: this description does not prevent \emph{a priori} a single initial zero $z_j^0$ to spontaneously split into several vortices $\lbrace a_k\rbrace$, because at $s=0$ the energy is not yet quantized (in the sense of \cite[Theorem~1.5]{BOS07-quant}). In fact initial splitting into two vortices can easily be ruled out, but it is not clear whether splitting into three or more vortices can occur.

However, note that in the setting of Corollary~\ref{c:nozero}, if there are no initial zeros, we can directly conclude that no later vortices appear. 
A complete proof of this fact would require adapting \cite{BOS08-survey} to our torus-based setting.

The main idea of this paper is that on the time scale considered, the effect of diffusion in the Ginzburg-Landau equation is dominated by the nonlinear effect. 
This means that the modulus of any initial data instantaneously (on the fast time scale $s=\eps^2 t/(\ln1/\eps)$) approaches  $1$ except possibly on small regions where the initial data is close to $0$. 
The methods are elementary and provide explicit pointwise estimates on $u(t,x)$, which directly imply the stated results.
To control diffusive effects, the key tool is Lemma~\ref{l:aux}, which is a type of Gronwall inequality (new to our best knowledge). 
The organization of the paper follows that of the presentation of the results which are proven in the same order in the consecutive sections.

\section{Zeros of $u$: proof of Theorem~\ref{t:zeros}}\label{s:zeros}

Denote by $\Phi\colon\R\times\R^2\to\R^2$ the flow of the ODE $y'=(1-|y|^2)y$, that is,
\begin{align*}
\partial_t\Phi=(1-|\Phi|^2)\Phi,\quad \Phi(0,X)=X,
\end{align*}
given explicitly by
\begin{align}\label{eq:Phi}
\Phi(t,X)=\frac{e^t X}{\sqrt{1+|X|^2(e^{2t}-1)}}.
\end{align}
We want to estimate how far $u$ is from 
\begin{align*}
v(t,x)=\Phi(t,g(t,x))
\end{align*} 
for some well-chosen map $g$ with $g(0,x)=u_0(x)$.
To this end we define $w=e^{-t}(u-v)$, so that
\begin{align*}
u=v +e^t w.
\end{align*}
Using the equations satisfied by $u$ and $\Phi$ we obtain
\begin{align}\label{eq:w}
\partial_t w -\e^2\Delta w &= -2(v\cdot w) v - |v|^2 w -e^{-t}\mathcal N(v,e^t w) -e^{-t}\mathcal R,\\
\mathcal N(v,X)&=|X|^2v +2 (v\cdot X) X +|X|^2 X,\nonumber\\
\mathcal R&=\partial_t v -\e^2\Delta v -(1-|v|^2)v\nonumber\\
&=D\Phi(t,g)(\partial_t g -\e^2\Delta g) -\e^2 D^2\Phi(t,g)\nabla g\cdot \nabla g\nonumber
\end{align}
In view of \eqref{eq:w}, it is natural to choose, as in \cite{chen04},
$g(t)=e^{\e^2t\Delta}u_0$, that is, $g$ solves
\begin{align}\label{eq:g}
\partial_t g -\e^2\Delta g=0,\quad g(0,x)=u_0(x),
\end{align}
and therefore
\begin{align}\label{eq:R}
\mathcal R&= -\e^2 D^2\Phi(t,g)\nabla g\cdot \nabla g.
\end{align}
The rest of the article is devoted to obtaining good pointwise estimates on $e^tw=u-v$.

\begin{lemma}\label{l:maxpple}
If $w$ solves
\begin{align*}
\partial_t w -\e^2 \Delta w = -2(v\cdot w) v - |v|^2 w +F,\qquad t>0,\; x\in\Omega,
\end{align*}
with $w(0,x)=0$, then
\begin{align*}
\|w(t)\|_{L^\infty}\leq \int_0^t \|F(s)\|_{L^\infty}\, ds.
\end{align*}
\end{lemma}
\begin{proof}[Proof of Lemma~\ref{l:maxpple}]
Multiplying the equation by $w/|w|$ we obtain
\begin{align*}
\partial_t |w| &=\e^2 \frac{w}{|w|}\cdot\Delta w -|v|^2|w|-2\frac{(v\cdot w)^2}{|w|}+F\cdot\frac{w}{|w|}\\
&\leq \e^2 \frac{w}{|w|}\cdot\Delta w +|F| \\
&\leq \e^2 \Delta |w| +|F|,
\end{align*}
so by comparison principle we have $|w|\leq \rho$ where $\rho$ solves $\partial_t\rho-\e^2\Delta\rho=|F|$ and $\rho(0,x)=0$, that is, $\rho(t)=\int_0^t e^{\e^2(t-s)\Delta}|F(s)|\, ds$, where $e^{t\Delta}$ denotes the heat semigroup on the torus $\tor$. Since the $L^\infty$-norm is nonincreasing under the action of that semigroup, we deduce the announced bound.
\end{proof}

We apply Lemma~\ref{l:maxpple} to our map $w$ and $F=-e^{-t}\mathcal N(v,e^tw) -e^{-t}\mathcal R$. 
We have $|g|\leq |u_0|\lesssim 1$,
so $|v|\lesssim 1$ and $|\mathcal N(v,X)|\lesssim |X|^2$. 
Thus we obtain
\begin{align}\label{eq:boundwNR}
\|e^t w(t)\|_{L^\infty}\lesssim \int_0^t e^{(t-s)}\|e^s w(s)\|^2_{L^\infty}\, ds +\int_0^t e^{t-s}\|\mathcal R(s)\|_{L^\infty}\, ds
\end{align}
Recall 
\begin{align*}
\mathcal R&=-\e^2 D^2\Phi(t,g)\nabla g\cdot \nabla g,
\end{align*}
and $|\nabla g|\leq |\nabla u_0|\lesssim 1$, hence
\begin{align*}
\|\mathcal R(t)\|_\infty\lesssim \e^2  \sup_{|X|\lesssim 1}|D^2\Phi(t,X)|.
\end{align*}
Direct calculation gives
\begin{align*}
|D^2\Phi(t,X)|&\lesssim \frac{e^t|X|(e^{2t}-1)}{(1+|X|^2(e^{2t}-1))^{3/2}} \\
&=e^t (e^{2t}-1)^{1/2}\frac{(|X|^2(e^{2t}-1))^{1/2}}{(1+|X|^2(e^{2t}-1))^{3/2}}\\
&\lesssim e^t (e^{2t}-1)^{1/2},
\end{align*}
so
\begin{align*}
\int_0^t e^{t-s}\|\mathcal R(s)\|_{L^\infty}\, ds &\lesssim
\e^2  e^t\int_0^t(e^{2s}-1)^{1/2}\, ds\\
&\lesssim \e^2 e^t (e^{2t}-1)^{1/2},
\end{align*}
where we have used
\begin{align*}
\int_0^t(e^{2s}-1)^{1/2}\, ds & =\int_0^{(e^{2t}-1)^{1/2}}\frac {x^2}{1+x^2}\, dx \\
& =(e^{2t}-1)^{1/2}-\arctan((e^{2t}-1)^{1/2}) \\
& \leq (e^{2t}-1)^{1/2}.
\end{align*}
Plugging this into \eqref{eq:boundwNR} we deduce
\begin{align}\label{eq:boundwimplicit}
\|e^t w(t)\|_{L^\infty}\lesssim \int_0^t e^{(t-s)}\|e^s w(s)\|^2_{L^\infty}\, ds +\e^2  e^t (e^{2t}-1)^{1/2}.
\end{align}

\begin{lemma}\label{l:aux}
Assume $f,h$ are continuous positive functions on $(0,\infty)$ satisfying 
\begin{align*}
&\limsup_{t\searrow 0}\frac{f(t)}{h(t)}\leq 1,\\
\text{and }&
f(t)\leq c\int_0^t e^{t-s}f(s)^2\, ds +h(t)\qquad\forall t>0,
\end{align*}
for some constant $c>0$. If $T>0$ is such that
\begin{align*}
\sup_{0<t<T} \int_0^t e^{t-s}\frac{h(s)}{h(t)}h(s)\, ds\leq \frac{1}{8c},
\end{align*}
then
\begin{align*}
f(t)\leq 2h(t)\qquad\forall t\in (0,T).
\end{align*}
If in addition $h$ is nondecreasing it suffices to check that
\begin{align*}
\int_0^T e^{T-s} h(s)\, ds \leq \frac{1}{8c}.
\end{align*}
\end{lemma}
\begin{proof}[Proof of Lemma~\ref{l:aux}]
For all $t>0$ we have
\begin{align*}
\frac{f(t)}{h(t)}&\leq c\int_0^t e^{t-s}\left(\frac{f(s)}{h(s)}\right)^2 \frac{h(s)}{h(t)}h(s)\, ds +1\\
&\leq c \int_0^t e^{t-s}\frac{h(s)}{h(t)}h(s)\, ds\, F(t)^2 +1,
\end{align*}
where
\begin{align*}
F(t)=\sup_{0<s<t}\frac{f(s)}{h(s)}.
\end{align*}
Since $F$ is nondecreasing and using that the integral in the right-hand side is bounded by $1/8c$ for $0<t<T$ we deduce
\begin{align*}
F(t)\leq \frac 18 F(t)^2 +1\qquad\forall t\in (0,T),
\end{align*}
so $F$ takes values into
\begin{align*}
\left\lbrace x\in \R\colon \frac {x^2}{8}-x+1 \geq 0\right\rbrace =
(-\infty,4-2\sqrt 2]
\cup [4+2\sqrt 2,+\infty).
\end{align*}
Since $F$ is continuous on $(0,T)$ and $F(0^+)\leq 1 < 4-2\sqrt 2$ we deduce that $F(t)\leq 4-2\sqrt 2\leq 2$ for all $t\in (0,T)$.
\end{proof}

We apply Lemma~\ref{l:aux} to 
\begin{align*}
f(t)=\|e^t w(t)\|_{L^\infty},\quad h(t)=A\e^2 e^t (e^{2t}-1)^{1/2},
\end{align*}
where $A\geq 1$ is the   constant hidden in the sign $\lesssim$ in \eqref{eq:boundwimplicit}.
By Lemma \ref{l:maxpple} the map $w$ satisfies $\|e^t w(t)\|_{L^\infty}\lesssim t$
so $\limsup_{0^+}(f/h)=0$, and thanks to \eqref{eq:boundwimplicit} we deduce that
\begin{align}\label{eq:unifboundw}
&\|e^t w(t)\|_{L^\infty}\lesssim \e^2   
e^t (e^{2t}-1)^{1/2}
\\
&\text{for }
0\leq t\leq T=\ln\frac{1}{\e}- \ln(16A^2),\nonumber
\end{align}
since $h$ is nondecreasing and for this value of $T$ we have
\begin{align*}
8A\int_0^T e^{T-s} h(s)\, ds & \leq 8A^2\e^2  e^{2T}\leq  \frac 12.
\end{align*}
Estimate \eqref{eq:unifboundw} tells us that $u$ is close to $v$. Note in particular that \eqref{eq:unifboundw}  is valid up to $t=\ln\frac{1}{\eps}-\frac{1}{2}\ln\ln\frac{1}{\eps}$ 
 if $\e$ is small enough.
From \eqref{eq:unifboundw} we also deduce a bound on $\nabla w$, using again the equation \eqref{eq:w}.
\begin{lemma}\label{l:heatgrad}
If $w$ solves $\partial_t w-\e^2\Delta w=F$ with $w(0)=0$ in the torus $\tor$, then we have
\begin{align*}
\|\nabla w(t)\|_{L^\infty}\lesssim \frac{1}{\e}\int_0^t \frac{\|F(s)\|_{L^\infty}}{\sqrt{t-s}}\,ds.
\end{align*}
\end{lemma}
\begin{proof}[Proof of Lemma~\ref{l:heatgrad}]
We can consider $w$ and $F$ as periodic maps defined on $\R^2$, then $w$ is given by the  Duhamel formula
\begin{align*}
w(t)=\int_0^t H_{\e\sqrt{t-s}} * F(s) \, ds,
\end{align*}
where the convolution is on $\R^2$ and $H_\delta(x)=\delta^{-2}H(x/\delta)$,  $H(x)=(4\pi)^{-1}e^{-|x|^2/4}$. Therefore we have
\begin{align*}
\|\nabla w(t)\|_{L^\infty}&\lesssim \int_0^t \|\nabla H_{\e\sqrt{t-s}}\|_{L^1}\|F(s)\|_{L^\infty} \, ds,
\end{align*}
and the estimate follows from
\begin{align*}
\|\nabla H_{\e\sqrt{t-s}}\|_{L^1} \lesssim\frac{1}{\e\sqrt{t-s}}\|\nabla H\|_{L^1}.
\end{align*}
\end{proof}
Applying Lemma~\ref{l:heatgrad} to the equation \eqref{eq:w} satisfied by $w$ and using \eqref{eq:unifboundw} to estimate the right-hand side we obtain
\begin{align}\label{eq:boundgradw}
\|\nabla w\|_{L^\infty}&\lesssim  \e  \sqrt t (e^{2t}-1)^{1/2},
\end{align} 
for all $t\leq T=\ln\frac{1}{\e}- \ln(16A^2)$.

All assertions of
Theorem 1 will follow from the bounds \eqref{eq:unifboundw}-\eqref{eq:boundgradw} on $e^{t}w=u-v$ and the explicit expression of $v=\Phi(t,g)$. First we need to gather some information on $g=e^{\e^2t\Delta}u_0$. To that end we use the nondegeneracy assumption \eqref{eq:u0nondegen}.
It implies that $u_0$
 has a finite number of zeroes, all of degree $\pm 1$. We denote
\begin{align*}
\lbrace u_0=0\rbrace =\lbrace z_1^0,\ldots,z_N^0\rbrace.
\end{align*}
Since $g(t,x)=\tilde g(\e^2 t,x)$ where $\tilde g(\tilde t)=e^{\tilde t\Delta}u_0$ is $C^1$ in $[0,\infty)\times\tor$, we deduce that there exists $t_0,\beta_0,r_0>0$ such that, for all $t\leq t_0/\e^2$,
\begin{align*}
&|g(t)|+|\det(\nabla g(t))|\geq\frac{\alpha_0}{2},\\
& |g(t,x)|\geq\beta_0\qquad\text{for  }\dist(x,\lbrace z_j^0\rbrace)\geq r_0,\\
& g(t)\text{ is invertible and }|\nabla g(t)^{-1}|\lesssim 1\text{ on }B(z_j^0,r_0).
\end{align*}
In each disk $B(z_j^0,r_0)$, the map $g(t)$ has exactly one zero $\hat z_j(t)$,
so
\begin{align*}
\lbrace g(t)=0\rbrace =\left\lbrace \hat z_1(t),\ldots,\hat z_N(t)\right\rbrace,
\end{align*}
 and we have
\begin{align}\label{eq:modulusg}
\dist(\cdot,\lbrace \hat z_j(t)\rbrace) \lesssim |g(t)|  \lesssim \dist(\cdot,\lbrace \hat z_j(t)\rbrace) 
\end{align}
Thanks to the implicit function theorem, the maps $t\mapsto \hat z_j(t)$ are $C^1$, and
\begin{align*}
\frac{d}{dt}\hat z_j(t)=-\nabla g(t,\hat z_j)^{-1}\partial_t g(t,\hat z_j),
\end{align*}
hence
\begin{align*}
|\frac{d}{dt}\hat z_j|&\lesssim   \|\partial_t g\|_\infty
\lesssim \e^2   \|\Delta g\|_{L^\infty}.
\end{align*}
Viewing $g$ and $u_0$ as periodic maps defined on $\R^2$, $g$ is given by the formula
\begin{align*}
g(t)=H_{\e\sqrt t}*u_0,
\end{align*}
where the convolution is on $\R^2$ and $H_\delta(x)=\delta^{-2}H(x/\delta)$,  $H(x)=(4\pi)^{-1}e^{-|x|^2/4}$, so
\begin{align*}
\|\Delta g\|_{L^\infty}&\leq \|\nabla H_{\e\sqrt t}\|_{L^1}\|\nabla u_0\|_{L^\infty}
\lesssim \frac{1}{\e\sqrt t},
\end{align*}
and we infer
\begin{align}\label{eq:estimzj}
|\frac{d}{dt}\hat z_j|
&\lesssim \frac{\e}{\sqrt t},
\qquad
|z_j(t)-z_j^0|\lesssim \e\,\sqrt t.
\end{align}
Next we combine these properties of $g(t)$ with the explicit expression \mbox{$v=\Phi(t,g)$} and the bounds \eqref{eq:unifboundw}-\eqref{eq:boundgradw} on $e^tw=u-v$ to obtain the desired properties on $u$. We denote by $C>0$ a generic constant depending on $u_0$ and which may change from line to line. We start by bounding the modulus $|u|$ from below: using \eqref{eq:Phi} and \eqref{eq:unifboundw} we obtain, for $0\leq t\leq \ln(1/\e)-C$,
\begin{align*}
|u|
&\geq |v|-e^t|w| 
\geq \frac{e^t|g|}{\sqrt{1+|g|^2(e^{2t}-1)}}-C\e^2 e^{2t}
\\
&\geq \frac 12 \min(e^t|g|,1)-C\e^2e^{2t}.
\end{align*}
The last quantity is positive whenever $e^t|g|\geq 1$ and $e^{2t}<1/(2C\e^2)$, or $e^t|g|\leq 1$ and $|g|^2>2C\e^2$. Hence we deduce that
\begin{align*}
|u|>0\quad\text{ in }\left\lbrace |g|\geq C \e  \right\rbrace\quad\text{for }0\leq t\leq \ln\frac 1\e -C.
\end{align*}
In the case without initial zeros, this proves in particular Corollary~\ref{c:nozero}.
Moreover, combining this with \eqref{eq:modulusg} we have $|u(t)|>0$ outside 
the disks $B(\hat z_j(t),C\e)$. 
By homotopy invariance of the topological degree, $u(t)$ must have at least one zero  $z_j(t)\in B(\hat z_j(t),C\e)$.
Next we verify that this zero is unique.

Recall that $g(t)$ is invertible on $B(z_0,r)$, and maps $B(z_j(t),C\e)$ into $B(0,K\e)$, for some constant $K$ depending on $u_0$, thanks to \eqref{eq:modulusg}. The flow map $\Phi(t)=\Phi(t,\cdot)$ is invertible from $B(0,K\e)$ onto $B(0,R)$ given by $R=\Phi(t,K\e)$, with inverse $\Phi(t)^{-1}=\Phi(-t)$. Therefore $v(t)=\Phi(t)\circ g(t)$ is invertible on $B(z_j(t),C\e)$, and
\begin{align*}
\sup_{v(B(\hat z_j(t),C\e))}|\nabla v(t)^{-1}|\lesssim \sup_{|X|\leq \Phi(t,K\e)} |\nabla\Phi(-t,X)|.
\end{align*}
For $t\leq \ln\frac 1\e -C$ we have
\begin{align*}
\Phi(t,K\e)&
=\frac{e^tK\e}{\sqrt{1+K^2\e^2(e^{2t}-1)}}\leq e^tK\e\leq  \frac 12,
\end{align*}
provided $C$ is large enough, and, for $|X|\leq 1/2$,
\begin{align*}
\nabla\Phi(-t,X)
&=\frac{e^{-t}}{\sqrt{1-|X|^2(1-e^{-2t})}}
\left( I+ \frac{1-e^{-2t}}{1-|X|^2(1-e^{-2t})}X\otimes X\right),
\end{align*}
so we infer
\begin{align*}
\sup_{v(B(\hat z_j(t),C\e))}|\nabla v(t)^{-1}| \lesssim e^{-t}.
\end{align*}
We use this to show that $u(t)$ is invertible on $B(\hat z_j(t),C\e)$. Since the equation $y=u(x)=v(x)+e^tw(x)$ is equivalent to $x=v^{-1}(y-e^tw(x))$, it suffices to check that the map $F\colon x\mapsto v^{-1}(y-e^tw(x))$ is a contraction on $B(\hat z_j(t),C\e)$, for $|y|< \delta$. Here $\delta>0$ is a small constant such that $v^{-1}$ is well defined on $B(0,2\delta)$. Thanks to \eqref{eq:unifboundw} we have $|e^tw|\leq \delta$ provided $C$ is large enough, and 
\begin{align*}
\sup_{B(\hat z_j(t),C\e)}|\nabla F| \lesssim e^{-t}\|e^t \nabla w\|_\infty \lesssim \|\nabla w\|_\infty.
\end{align*}
Since $\|\nabla w\|_\infty\lesssim \e\sqrt{t}e^t$ thanks to \eqref{eq:boundgradw}, we 
deduce that $F$ is a contraction for
\begin{align*}
0\leq t\leq T_\e =\ln\frac 1\e -\frac 12\ln\ln\frac 1\e -C_0,
\end{align*}
if the constant $C_0$ is large enough, depending on $u_0$. By the above discussion this shows that $u(t)$ is invertible on $B(\hat z_j(t),C\e)$, and 
\begin{align*}
\left\lbrace u(t)=0\right\rbrace =\left\lbrace  z_1(t),\ldots, z_N(t)\right\rbrace,
\end{align*}
for some $z_j(t)\in B(\hat z_j(t),C\e)$. 
This proves Theorem~\ref{t:zeros}, except for its last assertion \eqref{eq:baddisksTe}. 
To verify \eqref{eq:baddisksTe}, we note that \eqref{eq:unifboundw} ensures $|u-v|=|e^t w|\leq 1/4$ for $t=T_\e$, so it suffices to check that $|v(T_\e,x)|\geq 3/4$ for $\dist(x,\lbrace z_j^0\rbrace )\gtrsim \e\sqrt{\ln(1/\e)}$. 
We have
\begin{align*}
|v(T_\e)|&=\frac{e^{T_\e}|g|}{\sqrt{1+|g|^2(e^{2T_\e}-1)}}
= \frac{1}{\sqrt{1+(1-|g|^2)e^{-2T_\e}|g|^{-2} }}\\
&=\frac{1}{\sqrt{1+(1-|g|^2)e^{C_0}\left(\frac{\e\sqrt{\ln (1/\e)}}{g}\right)^2 }}.
\end{align*}
If $|g|\geq M\e\sqrt{\ln (1/\e)}$ for some large enough $M>0$, we deduce 
\mbox{$|v(T_\e)|\geq 3/4$}. Thanks to \eqref{eq:modulusg} this implies that $|v(T_\e,x)|\geq 3/4$ for $\dist(x,\lbrace z_j^0\rbrace )\gtrsim \e\sqrt{\ln(1/\e)}$ and concludes the proof of Theorem~\ref{t:zeros}.

\section{Energy of $u$: proof of Theorem~\ref{t:energy}}\label{s:energy}

First, we seek to obtain more precise estimates for $w$ away from the bad disks  $B(z_j^0, C\e\sqrt{\ln(1/\e)})$.
 To this end we localize the equation by setting
\begin{align*}
\tilde w =\chi^2 w ,
\end{align*}
for some appropriate smooth cut-off function $0\leq \chi(x) \leq 1$, to be chosen later. From the equation \eqref{eq:w} satisfied by $w$ we deduce
\begin{align}\label{eq:tildew}
\partial_t\tilde w -\e^2\Delta\tilde w &=
-2(v\cdot\tilde w)v -|v|^2\tilde w 
-e^{-t}\chi^2\mathcal N(v,e^tw)
-e^{-t}\chi^2\mathcal R \\
&\quad
 -\e^2(\Delta\chi^2) w-2\e^2\nabla\chi^2\cdot \nabla w.
\nonumber
\end{align}
Applying Lemma~\ref{l:maxpple} to the equation \eqref{eq:tildew} satisfied by $\tilde w$, and using \eqref{eq:unifboundw}-\eqref{eq:boundgradw}  to estimate the two last terms, we deduce
\begin{align*}
\|e^t\tilde w\|_{L^\infty} &\lesssim
\int_0^t e^{t-s}\|e^s\tilde w\|_{L^\infty}^2\, ds  
+ \int_0^t e^{t-s} \|\chi^2\mathcal R(s)\|_{L^\infty}\, ds \\
&\quad
+(\e\sqrt t \|\nabla\chi\|_{L^\infty} +\e^2\|\nabla^2\chi\|_{L^\infty})\e^2 e^t(e^{2t}-1)^{1/2},
\end{align*}
for all $t\leq \ln(1/\e)- C$. Applying Lemma~\ref{l:aux} we therefore have
\begin{align}\label{eq:boundtildewR}
\|e^t\chi^2w\|_{L^\infty}&\lesssim \int_0^t e^{t-s} \|\chi^2\mathcal R(s)\|_{L^\infty}\, ds\\
&\quad
+(\e\sqrt t \|\nabla\chi\|_{L^\infty} +\e^2\|\nabla^2\chi\|_{L^\infty})\e^2  e^t(e^{2t}-1)^{1/2},
\nonumber
\end{align}
provided $t\leq \ln(1/\e)- C$ and  $\e\sqrt t \|\nabla\chi\|_{L^\infty} +\e^2\|\nabla^2\chi\|_{L^\infty}\leq 1$.
Using the properties \eqref{eq:modulusg} of $g$, the fact that $|\hat z_j(t)-z_j^0|\lesssim \,\e \sqrt{\ln (1/\e)}$ for $t\leq\ln (1/\e)$ thanks to \eqref{eq:estimzj}, and letting
\begin{align*}
D(x)=\mathrm{dist}(x,\lbrace z_j^0\rbrace),
\end{align*}
we have 
\begin{align*}
 D \lesssim |g| \lesssim  D\quad\text{ in }\left\lbrace D\geq M   \e \sqrt{\ln \frac 1\e } \right\rbrace,
\end{align*}
for $t\leq \ln(1/\e)$. Here $M>0$ is a large constant that depends only on $u_0$.
 Since $|\nabla g|\lesssim 1$, recalling the explicit formulas \eqref{eq:R} and \eqref{eq:Phi} we deduce
\begin{align*}
|\mathcal R |&\lesssim \e^2 
\frac{e^t|g|(e^{2t}-1)}{(1+|g|^2(e^{2t}-1))^{3/2}} \\
&\lesssim 
\e^2  e^t(e^{2t}-1)
\frac{ D}{(1+C^{-2}D^2(e^{2t}-1))^{3/2}} ,
\end{align*}
in $\lbrace D\geq M\e\sqrt{\ln (1/\e)}\rbrace$ for $t\leq \ln(1/\e) - C$.
Therefore, choosing cut-off functions $\chi$ satisfying
\begin{align*}
\mathbf 1_{ 2\lambda \leq D \leq 3\lambda}\leq \chi\leq \mathbf 1_{ \lambda \leq D \leq 4\lambda},\quad |\nabla\chi|\lesssim \frac 1\lambda,\;|\nabla^2\chi|\lesssim \frac 1{\lambda^2},
\end{align*}
for some $\lambda \gtrsim \e\sqrt{\ln(1/\e)}$, from
 \eqref{eq:boundtildewR} we infer
 \begin{align}\label{eq:boundwD1}
|e^t w|&\lesssim  \e^2  e^t D\int_0^t 
\frac{e^{2s}-1}{(1+C^{-2}D^2(e^{2s}-1))^{3/2}}
\, ds\\
&\quad
+\frac{\e}{D}\sqrt{1+t}\,
\e^2 e^t(e^{2t}-1)^{1/2}
\nonumber
 \end{align}
in $\lbrace D\geq M  \e\sqrt{\ln(1/\e)}\rbrace$ for $t\leq \ln(1/\e)- C$ and $\e$ small enough.

For any $\alpha\in (0,1/2)$ we have
\begin{align*}
&\int_0^t \frac{e^{2s}-1}{(1+\alpha^2(e^{2s}-1))^{3/2}}\, ds\\
&
=\frac{1}{\alpha^2}
\int_1^{(1+\alpha^2(e^{2t}-1))^{1/2}}
\frac{(x^2-1)}{x^2(x^2-1+\alpha^2)}\, dx
\leq \frac{1}{\alpha^2}\int_1^\infty\frac{dx}{x^2},
\end{align*}
thanks to the change of variable $x=(1+\alpha^2(e^{2s}-1))^{1/2}$. Hence from \eqref{eq:boundwD1} we deduce
\begin{align}\label{eq:boundwD2}
\frac{1}{\e}|e^t w|&\lesssim\,\e\, e^t\,
\frac{\sqrt{1+t}}{D}
 \end{align}
in $\lbrace D\geq M  \e\sqrt{\ln(1/\e)}\rbrace$ for $t\leq \ln(1/\e)- C$ and $\e$ small enough.
Using Lemma~\ref{l:heatgrad} and \eqref{eq:boundwD2} to estimate the right-hand side of \eqref{eq:tildew}, we also obtain gradient bounds 
\begin{align}\label{eq:boundgradwD}
e^t|\nabla w|
&\lesssim \e\, e^t\,\sqrt t \,\frac{\sqrt{1+t}}{D} 
\end{align}
in $\lbrace D\geq M\e\sqrt{\ln(1/\e)} \rbrace$ for $t\leq \ln(1/\e)- C$ and $\e$ small enough.

Next we refine these estimates by including the effect of the second term $-|v|^2\tilde w$ in the right-hand side of \eqref{eq:tildew}. 
We choose as above  a cut-off function $\chi$ supported in $\lbrace \lambda \leq D \leq 4\lambda\rbrace$, with $|\nabla \chi|\lesssim \lambda^{-1}$ and $|\nabla^2\chi|\lesssim \lambda^{-2}$, for some $\lambda\geq M \e\sqrt{\ln\frac{1}{\e}}$.
Combining \eqref{eq:unifboundw}-\eqref{eq:boundgradw} and \eqref{eq:boundwD2}-\eqref{eq:boundgradwD} to bound the two last terms in \eqref{eq:tildew}, we have
\begin{align}\label{eq:tildeF}
\partial_t\tilde w -\e^2\Delta\tilde w &=
-2(v\cdot\tilde w)v -|v|^2\tilde w +\widetilde F,\\
|\widetilde F|&\lesssim e^{-t} \|e^t\tilde w\|_{L^\infty}^2 + e^{-t} \|\chi^2\mathcal R\|_{L^\infty}
\nonumber\\
&\quad + \frac{\e^3}{\lambda}\sqrt {1+t} \,
\min\left(\frac{\sqrt{1+t}}{\lambda},(e^{2t}-1)^{\frac 12}\right).
\nonumber
\end{align}
Arguing as in the proof of Lemma~\ref{l:maxpple} but retaining the second term in the right-hand side of \eqref{eq:tildeF}, we have
\begin{align*}
\partial_t |\tilde w| +|v|^2 |\tilde w | -\e^2 \Delta |w|\leq |\widetilde F|.
\end{align*}
In the support of $\chi$ we have
\begin{align*}
|v|^2 &=\frac{e^{2t}|g|^2}{1+|g|^2(e^{2t}-1)} 
=1-\frac{1-|g|^2}{1+|g|^2(e^{2t}-1)}\\
&
\geq \max\left(1-\frac{e^{-2t}}{C^{2}\lambda^{2}},0\right)
\end{align*}
hence
\begin{align*}
\partial_t |\tilde w| +\max\left(1-\frac{e^{-2t}}{C^{2}\lambda^{2}},0\right) |\tilde w | -\e^2 \Delta |w|\leq |\widetilde F|.
\end{align*}
We rewrite this as
\begin{align*}
\partial_t\,  e^{h(t)}|\tilde w| -\e^2 \Delta \, e^{h(t)}|\tilde w| \leq e^{h(t)} |\widetilde F|,
\end{align*}
where
\begin{align*}
h(t)&=\int_0^t  \max\left(1-\frac{e^{-2t}}{C^{2}\lambda^{2}},0\right)\, ds \\
& =\begin{cases}
0 & \text{ for }0<t<t_\lambda,\\
t-t_\lambda -\frac{1}{2C^2\lambda^2}(e^{-2t}-e^{-2t_\lambda})& \text{ for }t>t_\lambda,
\end{cases}
\end{align*}
where $t_\lambda=\ln(1/(C\lambda))$ is such that $1-e^{-2t_\lambda}/(C^2\lambda^2)=0$.
Arguing again as in Lemma~\ref{l:maxpple} we deduce
\begin{align*}
\|\tilde w(t)\|_{L^\infty}&\leq \int_0^t e^{h(s)-h(t)}\|\widetilde F(s)\|_{L^\infty}\, ds
\\
&=\int_0^{t_\lambda} \|\widetilde F(s)\|_{L^\infty}\, ds\\
& \quad +\int_{t_\lambda}^t e^{s-t} e^{\frac{1}{2C^2\lambda^2}(e^{-2t}-e^{-2s}) }\, \|\widetilde F(s)\|_{L^\infty}\, ds\\
&\leq 
\int_0^{t_\lambda} \|\widetilde F(s)\|_{L^\infty}\, ds
 +C\int_{t_\lambda}^t e^{s-t} \, \|\widetilde F(s)\|_{L^\infty}\, ds
\end{align*}
hence, from the bound on $\widetilde F$ in \eqref{eq:tildeF}, and estimating the term $\chi^2\mathcal R$ exactly as before (because the worst term in \eqref{eq:tildeF} is the last one anyway),
\begin{align*}
 \|e^t\tilde w(t)\|_{L^\infty} 
& \lesssim \int_0^t e^{t-s}\|e^s\tilde w(s)\|_{L^\infty}^2\, ds 
+ \int_0^{t} e^{t-s}\|\chi^2\mathcal R(s)\|_{L^\infty}\, ds\\
&\quad
+\frac{\e^3}{\lambda}e^t\int_0^{t_\lambda}\sqrt{1+s}(e^{2s}-1)^{1/2}\, ds  
+\frac{\e^3}{\lambda^2}\int_{t_\lambda}^t(1+s)\, e^s\, ds\\
&\lesssim 
\int_0^t e^{t-s}\|e^s\tilde w(s)\|_{L^\infty}^2\, ds 
+ \frac{\e^2}{\lambda} e^t 
+\frac{\e^3}{\lambda^2}e^t(1+t).
\end{align*}
Applying Lemma~\ref{l:aux} we obtain
\begin{align}\label{eq:boundwD3}
\frac 1\e |e^t  w | \lesssim \e \, e^{t}\, \frac{ 1+(\e/D)(1+t)}{D},
\end{align}
and with the help of Lemma~\ref{l:heatgrad} the gradient bound
\begin{align}\label{eq:boundgradwD2}
|e^t \nabla w|& \lesssim  \e\, e^t\,\sqrt t \, \frac{ 1+(\e/D)(1+t)}{D}.
\end{align}
These bounds are valid
in $\lbrace D\geq M  \e\sqrt{\ln(1/\e)}\rbrace$ for $t\leq \ln(1/\e)- C$ and $\e$ small enough.
Using
\begin{align*}
&\int_{\lbrace D\geq M\e\sqrt{\ln\frac{1}{\e}} \rbrace}\frac{ 1+(\e^2/D^2)(1+t)^2}{D^2}\, dx \\
&\lesssim  \int_{\e\sqrt{\ln\frac{1}{\e}}}^1 \frac{dr}{r} 
+\,\e^2(1+t)^2 \int_{\e\sqrt{\ln\frac{1}{\e}}}^1 \frac{dr}{r^3}\lesssim \ln\frac 1\e,
\end{align*}
and $u-v=e^tw$, we deduce the  energy bounds
\begin{align*}
&\int_{\lbrace D\geq M\e\sqrt{\ln\frac{1}{\e}} \rbrace}\frac{|u-v|^2}{\e^2}\, dx\lesssim \e^2 e^{2t}\ln\frac 1\e,
\\
&\int_{\lbrace D\geq M\e\sqrt{\ln\frac{1}{\e}} \rbrace} |\nabla u-\nabla v|^2\, dx \lesssim \e^2 e^{2t}t\ln\frac 1\e,
\end{align*}
and, using \eqref{eq:unifboundw}-\eqref{eq:boundgradw} to estimate the contributions from
$\lbrace D\lesssim \e\sqrt{\ln(1/\e)}\rbrace$,
\begin{align}\label{eq:energyu-v}
&\int_{\Omega} \left(
 |\nabla u-\nabla v|^2 + \frac{|u-v|^2}{\e^2}\right)\, dx\lesssim \e^2 e^{2t}t\ln\frac 1\e.
\end{align}
Note that this upper bound is $\lesssim \ln(1/\e)$ at $t=T_\e\leq\ln(1/\e\sqrt{\ln(1/\e)})$.
Next, we derive energy bounds for $v$.
We have
\begin{align*}
1-|v|^2&=1-\frac{e^{2t}|g|^2}{1+|g|^2(e^{2t}-1)}=\frac{1-|g|^2}{1+|g|^2(e^{2t}-1)}\\
&\leq \frac{1}{1+|g|^2(e^{2t}-1)},
\end{align*}
and since $|g|$ is of the same order as $\dist(\cdot,\lbrace \hat z_j(t)\rbrace)$ thanks to \eqref{eq:modulusg}
we deduce
\begin{align*}
\frac{1}{\e^2}\int_\Omega (1-|v|^2)^2\, dx 
&\leq \frac{1}{\e^2} \int_\Omega\frac{1}{(1+|g|^2(e^{2t}-1))^2}\, dx \\
&\lesssim 
\frac{1}{\e^2} \int_0^1 \frac{1}{(1+C^{-2}r^2(e^{2t}-1))^2}\, r\, dr \\
&\lesssim \frac{1}{\e^2}\frac{1}{ e^{2t}-1}
\end{align*}
We also have
\begin{align*}
|\nabla v|&=|D_X\Phi(t,g)\nabla g|
\lesssim \frac{e^t}{(1+|g|^2(e^{2t}-1))^{1/2}},
\end{align*}
hence
\begin{align*}
\int_\Omega |\nabla v|^2&\lesssim  e^{2t}\int_0^1 \frac{1}{1+C^{-2}r^2(e^{2t}-1)}\, r\, dr \\
&\lesssim \frac{e^{2t}}{e^{2t}-1}\ln(1+C^{-2}(e^{2t}-1))
\end{align*}
Gathering the above and recalling $T_\e=\ln(1/\e)-\frac{1}{2}\ln\ln(1/\e)-C_0$, we obtain
\begin{align*}
\int_\Omega \left( |\nabla v|^2+ \frac{1}{\e^2}(1-|v|^2)^2\right)\, dx \lesssim \ln\frac{1}{\eps}\qquad \text{at }t=T_\e.
\end{align*}
Combining this with the bounds \eqref{eq:energyu-v}  concludes the proof of Theorem~\ref{t:energy}.

\section{Jacobian of $u$ : proof of Theorem~\ref{t:jac}}

Define $\mathfrak u_\e(x)=u(T_\e,x)$, where $T_\e=\ln(1/\e)-(1/2)\ln\ln(1/\e)-C_0$ for a large enough constant $C_0$. 
We consider the jacobian
\begin{align*}
J\mathfrak u_\e =\det(\nabla \mathfrak u_\e),
\end{align*}
and show, as $\e\to 0$, the convergence
\begin{align}\label{eq:convJuTe}
J\mathfrak u_\e \to \pi\sum_{j=1}^N \hat d_j \delta_{z_j^0},
\end{align}
in the sense of distributions, where $\hat d_j\in\lbrace \pm 1\rbrace$ is the topological degree of $u_0$ at $z_j^0$.

Note that one can check, by direct calculation, that $Jv(T_\e)$ converges to this  sum of Dirac masses. But the bounds we have obtained on $e^tw=u-v$ are not enough to directly infer \eqref{eq:convJuTe}. Instead we invoke the compactness result of \cite[Theorem~3.1]{JS02-jac}:
thanks to the energy bound
\begin{align*}
E_\e(\mathfrak u_\e)\lesssim \ln\frac 1\e,
\end{align*}
there exists a sequence $\e_n\to 0$, integers $\tilde d_k\in\mathbb Z\setminus\lbrace 0\rbrace$ and distinct points $a_k\in\mathbb T^2$ such that
\begin{align*}
J\mathfrak u_{\e_n} \to \pi\sum_{k=1}^M \tilde d_k \delta_{a_k}.
\end{align*}
We show next that we must have $M=N$,  $\lbrace a_k\rbrace =\lbrace z_j^0\rbrace$,
without loss of generality $a_j=z_j^0$ for $j=1,\ldots,N$, and 
$\tilde d_j=\hat d_j$. 
Therefore the limit is unique and this proves  \eqref{eq:convJuTe}.

First we prove that $\lbrace a_k\rbrace \subset \lbrace z_j^0\rbrace$.
This is a consequence of the bounds obtained above on the map $u$, 
and the fact that the limit of $J\mathfrak u_{\e}$ provides a lower bound for $E_\e(\mathfrak u_\e)/\ln(1/\e)$ \cite[Theorem~4.1]{JS02-jac}. By that lower bound, for any $\delta>0$ we must have
\begin{align*}
E_{\e_n}(\mathfrak u_{\e_n};B(a_k,\delta))\geq \pi |d_k|\ln\frac 1{\e_n} +o\left(\ln\frac 1{\e_n}\right),
\end{align*}
as $n\to\infty$. Note that $|d_k|\geq 1$.
Therefore, to show that $\lbrace a_k\rbrace \subset \lbrace z_j^0\rbrace$ it suffices to obtain an upper bound of the form
\begin{align*}
E_\e(\mathfrak u_\e;B(a,\delta)) \leq\frac\pi 2\ln\frac 1\e\qquad\text{for }\e\ll 1,
\end{align*}
for any $a\notin \lbrace z_j^0\rbrace$ and some $\delta>0$.
Recall that we have $u=v+e^t w$, and the pointwise bounds  from \S~\ref{s:energy},
\begin{align*}
|\nabla v|^2+\frac{1}{\e^2}(1-|v|^2)^2&
\lesssim
\frac{1}{1+D^2 e^{2t}} \left( e^{2t} +\frac{1}{\e^2}\frac{1}{1+D^2e^{2t}}\right)
\\
|\nabla e^t w|^2  +\frac{1}{\e^2}|e^tw|^2
&
\lesssim \e^2e^{2t}\frac{t^2}{D^2},
\end{align*}
in $\lbrace D\geq C \e\ln^{1/2}(1/\e)\rbrace$ and for $1\leq t\leq\ln(1/\e)-C_0$.  
For $t=T_\e=\ln(1/\e)-(1/2)\ln\ln(1/\e)-C_0$ we deduce
\begin{align*}
|\nabla u|^2+\frac{1}{\e^2}(1-|u|^2)^2&
\lesssim
\frac{1+(\e/D)^2 e^{4C_0}\ln^2(1/\e)+e^{-2C_0}\ln(1/\e)}{D^2},
\end{align*}
in $\lbrace D\geq C \e\ln^{1/2}(1/\e)\rbrace$.
Hence at time $t=T_\e$, for any $\delta\geq \e\ln(1/\e)$ and for $\dist(a,\lbrace z_j^0\rbrace)\geq 4\delta$, we have
\begin{align*}
E_\e(\mathfrak u_\e;B(a,\delta))
&\lesssim 1 + \left(\frac{\e}{\delta}\right)^2e^{4C_0}\ln^2\frac 1\e +e^{-2C_0}\ln\frac 1\e \leq  \frac\pi 2\ln\frac 1\e,
\end{align*}
for $\e\ll 1$, 
provided $C_0$ is chosen large enough.
By the above discussion, this implies that $\lbrace a_k\rbrace \subset\lbrace z_j^0\rbrace$.

Therefore we may write
\begin{align}\label{eq:convJutilde}
J\mathfrak u_{\e_n} \to \pi\sum_{j=1}^N \tilde d_j \delta_{z_j^0},
\end{align}
where $\tilde d_j \in\mathbb Z$. 
Here we allow the possibility that $\tilde d_j=0$ because we have not proven yet that $\lbrace z_j^0\rbrace\subset\lbrace a_k\rbrace$.
 To prove \eqref{eq:convJuTe}, it suffices to show that $\tilde d_j=\hat d_j$. To that end, note that for all small $\e>0$ we have
\begin{align*}
\frac{1}{\pi}\int_{B(z_j^0,r)} \det(\nabla \mathfrak u_\e)
=\deg(\mathfrak u_\e,\partial B(z_j^0,r))=\hat d_j,
\end{align*}
for any small $r>0$ and $j=1,\ldots,N$. This is because $t\mapsto u(t)$ is smooth and $u(t)$ doesn't vanish on $\partial B(z_j^0,r)$ for small $\e>0$ and all $t\in [0,T_\e]$, so the degree of $\mathfrak u_\e=u(T_\e)$ is equal to the degree of $u_0=u(0)$, which is $\hat d_j$ by definition. Therefore, testing \eqref{eq:convJutilde} with a test function $\varphi\approx \mathbf 1_{B(z_j^0,r)}$, we obtain $\tilde d_j\approx \hat d_j$, hence $\tilde d_j=\hat d_j$ because these are integers. This concludes the proof of \eqref{eq:convJuTe}.

\bibliographystyle{alphaa}
\bibliography{GLflow}

\begin{thebibliography}{BCPS95}

\bibitem[BBH94]{BBH}
{\sc Bethuel, F., Br{\'e}zis, H., and H{\'e}lein, F.}
\newblock {\em Ginzburg-{Landau} vortices}, volume~13 of {\em Prog. Nonlinear
  Differ. Equ. Appl.}
\newblock Boston, MA: Birkh{\"a}user, 1994.

\bibitem[BCPS95]{BCPS95}
{\sc Bauman, P., Chen, C.-N., Phillips, D., and Sternberg, P.}
\newblock Vortex annihilation in nonlinear heat flow for {Ginzburg}-{Landau}
  systems.
\newblock {\em Eur. J. Appl. Math.}, 6(2):115--126, 1995.

\bibitem[BOS05]{BOS05-coll}
{\sc Bethuel, F., Orlandi, G., and Smets, D.}
\newblock Collisions and phase-vortex interactions in dissipative
  {Ginzburg}-{Landau} dynamics.
\newblock {\em Duke Math. J.}, 130(3):523--614, 2005.

\bibitem[BOS07a]{BOS07-dyn}
{\sc Bethuel, F., Orlandi, G., and Smets, D.}
\newblock Dynamics of multiple degree {Ginzburg}-{Landau} vortices.
\newblock {\em Commun. Math. Phys.}, 272(1):229--261, 2007.

\bibitem[BOS07b]{BOS07-quant}
{\sc Bethuel, F., Orlandi, G., and Smets, D.}
\newblock Quantization and motion law for {Ginzburg}-{Landau} vortices.
\newblock {\em Arch. Ration. Mech. Anal.}, 183(2):315--370, 2007.

\bibitem[BOS08]{BOS08-survey}
{\sc B{\'e}thuel, F., Orlandi, G., and Smets, D.}
\newblock On the {Cauchy} problem for phase and vortices in the parabolic
  {Ginzburg}-{Landau} equation.
\newblock In {\em Singularities in PDE and the calculus of variations. Selected
  papers of the CRM workshop, Montreal, Canada, July 17-- 21, 2006}, pages
  11--31. Providence, RI: American Mathematical Society (AMS), 2008.

\bibitem[Che04]{chen04}
{\sc Chen, X.}
\newblock Generation, propagation, and annihilation of metastable patterns.
\newblock {\em J. Differ. Equations}, 206(2):399--437, 2004.

\bibitem[JS98]{JS98-dyn}
{\sc Jerrard, R.~L., and Soner, H.~M.}
\newblock Dynamics of {Ginzburg}-{Landau} vortices.
\newblock {\em Arch. Ration. Mech. Anal.}, 142(2):99--125, 1998.

\bibitem[JS02]{JS02-jac}
{\sc Jerrard, R.~L., and Soner, H.~M.}
\newblock The {Jacobian} and the {Ginzburg}-{Landau} energy.
\newblock {\em Calc. Var. Partial Differ. Equ.}, 14(2):151--191, 2002.

\bibitem[Lin96]{lin96}
{\sc Lin, F.~H.}
\newblock Some dynamical properties of {Ginzburg}-{Landau} vortices.
\newblock {\em Commun. Pure Appl. Math.}, 49(4):323--359, 1996.

\bibitem[RS95]{RS95}
{\sc Rubinstein, J., and Sternberg, P.}
\newblock On the slow motion of vortices in the {Ginzburg}-{Landau} heat flow.
\newblock {\em SIAM J. Math. Anal.}, 26(6):1452--1466, 1995.

\bibitem[Ser07a]{ser07-II}
{\sc Serfaty, S.}
\newblock Vortex collisions and energy-dissipation rates in the
  {Ginzburg}-{Landau} heat flow.
\newblock {\em J. Eur. Math. Soc. (JEMS)}, 9(3):383--426, 2007.

\bibitem[Ser07b]{ser07-I}
{\sc Serfaty, S.}
\newblock Vortex collisions and energy-dissipation rates in the
  {Ginzburg}-{Landau} heat flow. {I}: {Study} of the perturbed
  {Ginzburg}-{Landau} equation.
\newblock {\em J. Eur. Math. Soc. (JEMS)}, 9(2):177--217, 2007.

\bibitem[Ser17]{ser17}
{\sc Serfaty, S.}
\newblock Mean field limits of the {Gross}-{Pitaevskii} and parabolic
  {Ginzburg}-{Landau} equations.
\newblock {\em J. Am. Math. Soc.}, 30(3):713--768, 2017.

\bibitem[SS04]{SS04}
{\sc Sandier, E., and Serfaty, S.}
\newblock Gamma-convergence of gradient flows with applications to
  {Ginzburg}-{Landau}.
\newblock {\em Commun. Pure Appl. Math.}, 57(12):1627--1672, 2004.

\end{thebibliography}

\end{document}